\documentclass[12pt]{amsart}
\usepackage{amsmath,amssymb,amsthm}
\usepackage{enumerate}
\usepackage{graphicx}
\usepackage{amsfonts}
\usepackage{bigints}
\usepackage{amsfonts}
\usepackage{mathptmx}      % use Times fonts if available on your TeX system
\usepackage{graphicx}
\usepackage{amssymb}
\usepackage{amsmath}
\usepackage{bigints}
\usepackage{mathptmx}      % use Times fonts if available on your TeX system
\newcommand{\overbar}[1]{\mkern 1.5mu\overline{\mkern-1.5mu#1\mkern-1.5mu}\mkern 1.5mu}

\theoremstyle{plain}%note that this is the default style (to learn)
\newtheorem{theorem}{Theorem}[section]
\newtheorem{proposition}[theorem]{Proposition}
\newtheorem{corollary}[theorem]{Corollary}

\theoremstyle{definition}
\newtheorem{definition}[theorem]{Definition}

\makeatletter
\@addtoreset{equation}{section}
\makeatother

\makeatletter
\newcommand{\Spvek}[2][r]{%
	\gdef\@VORNE{1}
	\left(\hskip-\arraycolsep%
	\begin{array}{#1}\vekSp@lten{#2}\end{array}%
	\hskip-\arraycolsep\right)}

\def\vekSp@lten#1{\xvekSp@lten#1;vekL@stLine;}
\def\vekL@stLine{vekL@stLine}
\def\xvekSp@lten#1;{\def\temp{#1}%
	\ifx\temp\vekL@stLine
	\else
	\ifnum\@VORNE=1\gdef\@VORNE{0}
	\else\@arraycr\fi%
	#1%
	\expandafter\xvekSp@lten
	\fi}
\makeatother
\begin{document}
	\title[Orthogonal polynomials on Cantor sets of zero Lebesgue measure]{Orthogonal polynomials on Cantor sets of zero Lebesgue measure}
	\author{G\"{o}kalp Alpan}
\address{Department of Mathematics, Bilkent University, 06800 Ankara, Turkey}
\email{gokalp@fen.bilkent.edu.tr}
\thanks{The author is supported by a grant from T\"{u}bitak: 115F199.}
\begin{abstract}
In this survey article, we review some results and conjectures related to orthogonal polynomials on Cantor sets. The main purpose of this paper is to emphasize the role of equilibrium measures in order to have a general theory of sufficiently good measures (measures that behave similarly to measures which are in the Szeg\H{o} class and the isospectral torus in the finite gap case) supported on totally disconnected subsets of $\mathbb{R}$. We present some open problems a number of which can be studied numerically.
	
\end{abstract}
\maketitle
\section{Introduction}
Throughout the article, by a measure $\mu$ we mean a unit Borel measure with an infinite compact support on $\mathbb{R}$. For such a measure $\mu$, we can find (see e.g. \cite{ase}) a sequence of polynomials $\{P_n(\cdot;\mu)\}_{n=1}^\infty$ such that $P_n(x;\mu)= \gamma_n x^n+\dots$ with $\gamma_n>0$ and  
\begin{equation*}
\int P_n(x;\mu) P_m(x;\mu) d\mu(x)=\delta_{mn}.
\end{equation*}
We call $P_n(\cdot;\mu)$ the $n$-th orthonormal polynomial for $\mu$. If we assume that $P_{-1}(\cdot;\mu):=0$ and $P_0(\cdot;\mu):=1$ then these polynomials satisfy a three term recurrence relation, that is there are two bounded sequences $(a_n)_{n=1}^\infty$ and $(b_n)_{n=1}^\infty$ such that for $n\geq 0$ we have
\begin{equation}\label{recur}
xP_n(x;\mu) = a_{n+1}P_{n+1}(x;\mu) + b_{n+1}P_n(x;\mu) + a_n P_{n-1}(x;\mu).
\end{equation}
Note that, $a_n>0$ and $b_n\in\mathbb{R}$  hold true for $n\in\mathbb{N}$ because of the imposed initial conditions. We define the $n$-th monic orthogonal polynomial by $p_n(x;\mu):= P_n(x;\mu)/\gamma_n$. 

Conversely, if we are given two bounded sequences $(a_n)_{n=1}^\infty$ and $(b_n)_{n=1}^\infty$ with $a_n>0$ and $b_n\in\mathbb{R}$ for $n\in\mathbb{N}$ then we can define a self-adjoint bounded operator $J^+: l^2(\mathbb{N})\rightarrow l^2(\mathbb{N})$ (we call these operators (one sided) Jacobi operators or Jacobi matrices) as follows where the matrix is represented in the standard basis:

\begin{equation}
J^+=\left( \begin{array}{ccccc}
b_1 & a_1 &0 & 0 &\ldots \\
a_1 & b_2 & a_2 & 0& \ldots \\
0& a_2 & b_3 & a_3 & \ldots \\
\vdots & \vdots & \vdots & \vdots&\ddots \\
\end{array} \right).
\end{equation}

By Favard's theorem, the scalar valued spectral measure of $J^+$ for the vector $(1,0,0,\ldots)^T$ is the measure which produces $(a_n)_{n=1}^\infty$ and $(b_n)_{n=1}^\infty$ in \eqref{recur}. Moreover this implies that the spectrum $\sigma(J^+)$ and the support of $\mu$ coincide. Because of this one-to-one correspondence between Jacobi matrices and measures we often use $J^+(\mu)$ and $J^+\left((a_n,b_n)_{n=1}^\infty\right)$ in the article.

Let $\delta_n$ be the normalized counting measure on the zeros of the $n$-th monic orthogonal polynomial for $\mu$. If there is a measure $\nu$ such that $\delta_n\rightarrow \nu$ in weak star sense then $\nu$ is called the density of states (DOS) measure for $\mu$ or for $J^+(\mu)$.

Similarly, one can also define two sided Jacobi matrices. For two bounded sequences $(a_n)_{n=-\infty}^\infty$ and $(b_n)_{n=-\infty}^\infty$ with $a_n\geq 0$ and $b_n\in\mathbb{R}$ for $n\in\mathbb{Z}$, the corresponding Jacobi matrix is defined as follows:

\begin{equation*}
J=\begin{pmatrix}
	\ddots & \ddots & \ddots \\ 
	& a_{-1} &b_0 & a_0    \\
	& & a_0 & b_1& a_1  \\
	& & & a_1 & b_2 & a_2 \\
	& & & & \ddots & \ddots & \ddots
\end{pmatrix}.
\end{equation*}

For the restriction $J^+\left((a_k,b_k)\right)_{k=n+1}^\infty$  of $J$ we use  $J^{+}_{\restriction_n}$. In order to denote $J^+\left((a_{n-k},b_{n-k+1})\right)_{k=1}^\infty$ we use $J^{-} _{\restriction_n}$. If $n=0$ we skip the subscript. In order these restrictions to be well defined, $a_k\neq 0$ or $a_{n-k}\neq 0$ should be valid respectively for $k\in\mathbb{N}$. By the DOS measure for $J$ we mean the DOS measure for $J^+$.

Theory of asymptotics of orthogonal polynomials are closely related to potential theory. For a general treatment of logarithmic potential theory, see e.g. \cite{Ransford}, \cite{saff}. Let us denote the logarithmic capacity by $\mathrm{Cap}(\cdot)$. A measure $\mu$ satisfying 
\begin{equation}\label{reg}
\lim_{n\rightarrow\infty} \|p_{n}(\cdot;\mu)\|_{L^2(\mu)}^{1/n}= \mathrm{Cap(supp(\mu))}
\end{equation}
is called regular in the sense of Stahl-Totik where $\|\cdot\|_{L^2(\mu)}$ denotes the Hilbert norm in $L^2(\mu)$ and $\mathrm{supp}(\mu)$ stands for the support of $\mu$. In \eqref{reg}, $\|p_{n}(\cdot;\mu)\|_{L^2(\mu)}$ can be replaced by $a_1\cdots a_n$ since these two quantities are equal. Moreover, for a regular measure $\mu$ with a non-polar compact support $K$, $\delta_n$ converges to the equilibrium measure of $K$ as $n\rightarrow\infty$, see e.g. Theorem 1.7 in \cite{simon1}. If a measure $\mu$ is regular then we use the notation $\mu\in{\mathrm{\bf{Reg}}}$. We say that $J^+(\mu)$ is regular if $\mu\in{\mathrm{\bf{Reg}}}$. Equivalent characterizations of \eqref{reg} and more on regular measures can be found in \cite{Stahl} and \cite{simon1}. 

There are some classes of measures for which the asymptotics stronger than \eqref{reg} hold. Let $K=\cup_{i=1}^n [\alpha_i, \beta_i]$ (we call these sets finite gap sets) where each $[\alpha_i,\beta_i]$ is a non-degenerate compact interval on $\mathbb{R}$. The Szeg\H{o} class of measures and the measures associated with the isospectral torus on $K$ are typical examples of such class of measures, see \cite{apt}, \cite{chriss}, \cite{Chris}, \cite{widom2}. If we replace a finite gap set by a Parreau-Widom set the concepts of Szeg\H{o} class and isospectral torus on this set still make sense but the theory is more complicated and less complete in this generality. We refer the reader to \cite{christiansen}, \cite{christi}, \cite{peher}, \cite{sodin}, \cite{volberg} for some results about orthogonal polynomials and Jacobi matrices on Parreau-Widom sets. 

If $L\subset \mathbb{R}$ is a zero Lebesgue measure Cantor set with positive capacity then we do not have a general theory of measures which are analogues of Szeg\H{o} class and isospectral torus on $L$. There are still interesting results -both analytic and numerical- and conjectures regarding what these two classes of objects may mean in this case. 

Orthogonal polynomials on zero Lebesgue measure sets are our main focus here. We review related results, conjectures and suggested definitions. We prove a couple of theorems on Parreau-Widom sets which also have some implications on zero measure case. We define Szeg\H{o} class and isospectral torus on more general sets and in our approach the concept of equilibrium measure is central.

The plan of the paper is as follows. In Section 2 and Section 3 we discuss the results which are valid on finite gap sets and Parreau-Widom sets concerning Szeg\H{o} class and isospectral torus. We also prove some new results in these sections, see Theorem \ref{dct}, Theorem \ref{IT} and Theorem \ref{szeg}. In Section 4, we propose definitions for Szeg\H{o} class and isospectral torus which are compatible with the definitions in finite gap sets and previously suggested definitions for zero measure case. In Section 5, we review some conjectures and results concerning orthogonal polynomials where the support of the prescribed measures are of zero Lebesgue measure.  In Section 6, we state several open problems related to these concepts. 

For a non-polar compact subset $K$ of $\mathbb{R}$ the Green function with a pole at infinity for $\overline{\mathbb{C}}\setminus K$ is denoted by $g_K$ and the equilibrium measure is denoted by $\mu_K$. We use the notation $|\cdot|$ for the Lebesgue measure of a set $L$ on $\mathbb{R}$. In order to denote the essential spectrum we use $\sigma_{\mathrm{ess}}$. Here, the essential spectrum is used for the set of all accumulation points of the spectrum.
\section{Isospectral Torus}

There are various ways to define the isospectral torus on a finite gap set $K$, see Section 3 in \cite{Chris2}. In order to discuss what this concept may mean on more general sets, we first give background information on related concepts.

Let $K$ be a non-polar compact subset of $\mathbb{R}$ which is regular with respect to the Dirichlet problem. Then $K$ is a Parreau-Widom set if $\sum_n g_K (c_n)<\infty$ where $\{c_n\}_n$ is the set of critical points of $g_K$. If $K$ is a Parreau-Widom set, we call $\overline{\mathbb{C}}\setminus K$ a Parreau-Widom domain. We remark that on a Parreau-Widom set $K$, the Lebesgue measure $dx_{\restriction K}$ restricted to $K$ and $d\mu_K$ are mutually absolutely continuous, see \cite{sodin}. In particular this implies that $|K|>0$. Regularity of $K$ with respect to the Dirichlet problem implies by Theorem 4.2.3 in \cite{Ransford} and Theorem 5.5.13 in \cite{Sim3} that $\mathrm{supp}(\mu_K)= K$. Therefore, 
\begin{equation}\label{ess}
|(x-a,x+a)\cap K|>0
\end{equation}
holds for all $x\in K$ and $a>0$.

As it was shown in \cite{peher}, \cite{sodin}, \cite{volberg} there are good and bad (these terms were used in \cite{volberg} in this fashion) Parreau-Widom sets. In this context, good and bad Parreau-Widom sets are classified in terms Direct Cauchy theorem (DCT), see \cite{christi} and \cite{yuditskii} for more on this issue.  We do not discuss what DCT means here since it is long and technical but we discuss recent results which involve this concept. 

For a given Jacobi operator $J=J(a_n,b_n)_{n=-\infty}^\infty$ let us denote $(a_n,b_n)_{n=-\infty}^\infty$ by $c_J$ for simplicity. We define the shift operator $S$ by 
\begin{equation*}
S(c_J)=S((a_n,b_n)_{n=-\infty}^\infty)= (a_{n+1},b_{n+1})_{n=-\infty}^\infty  
\end{equation*}

A bounded  $\mathbb{C}$-valued sequence $(d_n)_{n=-\infty}^\infty$ is called almost periodic if $\{(d_{n+k})_{n=-\infty}^\infty:\,\, k\in\mathbb{Z}\}$ is precompact in $l^\infty(\mathbb{Z})$. A Jacobi operator $J$ is called almost periodic if 
\begin{equation}\label{shift}
\{S^k(c_J):\,\, k\in\mathbb{Z}\}
\end{equation}

is precompact in $l^\infty(\mathbb{Z})\times l^\infty(\mathbb{Z})$ which is equipped with the metric 
$$d_l((a_n,b_n)_{n=-\infty}^\infty,(a^{\prime}_n,b^{\prime}_n)_{n=-\infty}^\infty)= \max\left(\sup_{n\in \mathbb{Z}}|a_n-a^\prime_n|,\sup_{n\in \mathbb{Z}}|b_n-b^\prime_n|\right) .$$

We denote the closure of the set given in \eqref{shift} by $\Omega_{c_J}$ and call it the hull of $c_J$. The hull can be made into a compact abelian group, with the composition as the group operation. Moreover, there is a unique invariant measure $m_{\Omega_{c_J}}$ which is also ergodic on $\Omega_{c_J}$, see Section 5.3 in \cite{teschl} for more details. This is equivalent to saying that the set of two sided Jacobi operators corresponding to the elements of the hull is a family of uniquely ergodic elements and actually from \eqref{it1} of Theorem \ref{bigtheo} below this family is strictly ergodic. See e.g. \cite{beckus} for more on ergodic Jacobi operators.

A one-sided Jacobi operator $J^+$ is called almost periodic if it is the restriction of a two sided almost periodic Jacobi operator $J$ to $\mathbb{N}$. 

Another equivalent characterization of almost periodicity is the following, see Appendix to Section 5.13 in \cite{Sim3}: For every $\varepsilon>0$ there is an $L\in\mathbb{N}$ such that for every $m\in\mathbb{Z}$ there is an $r$ such that $|m-r|<L$ implies that 
$d_l\left(S^m(c_J), S^r(c_J)\right)<\varepsilon$.

We list some properties of almost periodic sequences and Jacobi operators.
\begin{theorem}\label{bigtheo}
	\begin{enumerate}
		\item \label{it1} If $J$ is almost periodic and $c_{J^\prime}\in\Omega_{c_J}$ then $\Omega_{c_J}= \Omega_{c_{J^\prime}}$. Hence $\{S^n(x)\}_{n\in\mathbb{Z}}$ is dense in $\Omega_{c_J}$ for all $x\in\Omega_{c_J}$. 
		\item \label{it2}Let $(a_n)_{n=-\infty}^\infty$ be a real-valued almost periodic sequence with $\liminf_{n\rightarrow\infty}a_n= c$. Then for every $x=(x_n)_{n=-\infty}^\infty$ in the closure of $\{(a_{n+k})_{n=-\infty}^\infty:\,\, k\in\mathbb{Z}\}$ we have,
		\begin{equation}
		\inf_{n\in\mathbb{Z}} x_n= \liminf_{n\rightarrow\infty} x_n= \liminf_{n\rightarrow-\infty} x_n=c
		\end{equation}
		\item \label{it3}Let $J$ be almost periodic and $f:(a^\prime_n, b^\prime_n)_{n=-\infty}^\infty\rightarrow a^\prime_0$ for $(a^\prime_n, b^\prime_n)_{n=-\infty}^\infty\in\Omega_{c_J}$. If 
		\begin{equation}
		A(\Omega_{c_J}):=\int \log{f((a^\prime_n, b^\prime_n)_{n=-\infty}^\infty)} dm_{\Omega_{c_J}}>-\infty
		\end{equation}
		then $a_n^\prime>0$ for all $n\in\mathbb{Z}$ for $m_{\Omega_{c_J}}$ almost every $c^\prime$.
		\item \label{it4} If $J$ is almost periodic then there is a $\tilde{\nu}$ such that for all $c_{J^\prime}\in \Omega_{c_J}$ the DOS measure for $J^\prime(a^\prime_n,b^\prime_n)$ is $\tilde{\nu}$ provided that $a^\prime_n>0$ for all $n\in\mathbb{N}$. Moreover,
		\begin{equation}\label{essup}
		\sigma(J^\prime)=\sigma_{\mathrm{ess}}((J^{\prime})^+)=\mathrm{supp}(\tilde{\nu}).
		\end{equation}
	\end{enumerate}
\end{theorem}
\begin{proof}
	From the above discussion, \eqref{it1} and \eqref{it2} easily follow.  Proof of \eqref{it3} can be found in Section 3 of \cite{carmona}. In \eqref{it4}, $\sigma(J^\prime)=\mathrm{supp}(\tilde{\nu})$ can be found in Section 2 in \cite{lenz}. By Proposition 1.8 in \cite{simon1} we have $\mathrm{supp}(\tilde{\nu})\subset \sigma_{\mathrm{ess}}((J^{\prime})^+)$ in \eqref{it4}. But since $\sigma_{\mathrm{ess}}((J^{\prime})^+)\subset \sigma(J^\prime)$ by Lemma 3.7 in \cite{teschl} we have \eqref{essup}.
\end{proof}

For a given $c=(a_n,b_n)_{n=-\infty}^\infty$ with $a_n\neq 0$ for $n\geq 0$, let 
\begin{equation}
M_n(z)=\begin{pmatrix}
(z-b_n)/a_n & -{a_{n-1}/a_n}\\
1& 0
\end{pmatrix},
\end{equation}
and 
\begin{equation*}
M^{(n)}(z)= M_n(z)\cdots M_1(z),
\end{equation*}
for $z\in\mathbb{C}^+$. If $J$ is almost periodic and $A(\Omega_{c_J})>-\infty$ then there is a function $\tilde{\gamma}$ (see e.g. (3.6) in \cite{carmona}) such that 
$$\tilde{\gamma}(z)=\lim_{n\rightarrow\infty} \frac{1}{n} \log{\| M^{(n)}(z)\|}$$ for $m_{\Omega_{c_J}}$ almost every $c_{J^\prime}$. In particular, $\lim_{n\rightarrow\infty} \log{(a_1\cdots a_n)^{1/n}}=\lim_{n\rightarrow\infty}(1/n)\sum_{k=1}^n \log{a_k} = A(\Omega_{c_J})$ for $m_{\Omega_{c_J}}$ almost every $c_{J^\prime}$. Moreover, we also have 
$\tilde{\gamma}(z)= \int \log{|z-t|}d\tilde{\nu}(t)-A(\Omega_{c_J})>0$. In addition, 
\begin{equation*}
\gamma(x):= \lim_{y\rightarrow 0^+} \tilde{\gamma}(x+iy)
\end{equation*}
 exists for all $x\in\mathbb{R}$, see Section 2 in \cite{hur}. Let $\gamma:= \tilde{\gamma}$ on $\mathbb{C}^+$. Then the function $\gamma$ which is defined on $\mathbb{C}\cup \mathbb{R}$ is called the Lyapunov exponent. See e.g. Section 7 of \cite{simon1} and \cite{knill2} for more on these concepts.
 
The next result concerning regular measures will be used when we discuss Jacobi matrices on Julia sets. We have $\mathrm{Cap(supp}(\mu))= \mathrm{Cap(ess\,supp}(\mu)),$ see Section 1 of \cite{Sim3}. Thus, for an almost periodic $J$, we have $\mathrm{Cap}(\sigma(J))= \mathrm{Cap}(\sigma(J^+))$ by part \eqref{it4} of Theorem \ref{bigtheo}. If $\eta$ is regular and $\mathrm{Cap(supp}(\eta))>0$, then the DOS measure for $\eta$ is $\mu_{\mathrm{(supp}(\eta))}$, see \cite{Sim3}.

\begin{proposition}\label{prop}
Let $J(a_n,b_n)_{n=-\infty}^\infty$ be almost periodic and $\mathrm{Cap}(\sigma(J))>0$. Moreover, let $J^+$ be regular. Then $A(\Omega_{c_J})>-\infty$ and $\gamma=g_{\sigma(J)}$ on $\mathbb{C}^+$. As a corollary $\gamma=0$ quasi-everywhere on $\sigma(J)$. Besides, for $m_{\Omega_{c_J}}$ almost every $J^\prime$, $(J^\prime)^+$ is regular in the sense of Stahl-Totik.
\end{proposition}
\begin{proof}
	We have $\lim_{n\rightarrow\infty}(a_1\cdots a_n)^{1/n}=\mathrm{Cap}({\sigma(J)})$ by regularity and almost periodicity of $J^+$. Let $J^\varepsilon= J(a_n+\varepsilon,b_n)_{n=-\infty}^\infty$ for $\varepsilon>0$. Then $J^\epsilon$ is also almost periodic and $\liminf_{n\rightarrow\infty} (a_n+\varepsilon)>0$. This implies that $A(\Omega_{c_{J^\varepsilon}})>-\infty$. Therefore, $\limsup_{\epsilon\rightarrow 0^+}A(\Omega_{c_{J^\varepsilon}})\leq A(\Omega_{c_J})$ by Section 5 in \cite{hur}. Since $\liminf_{n\rightarrow\infty} (a_n+\varepsilon)>0$, by Theorem 7.1 (f) in \cite{simon1}, $\exp^{A(\Omega_{c_{J^\varepsilon}})}= \lim_{n\rightarrow\infty}((a_1+\varepsilon)\cdots (a_n+\varepsilon))^{1/n}$ holds. Thus, 
	\begin{equation}\label{yard}
	\exp^{A(\Omega_{c_{J}})}\geq \mathrm{Cap}({\sigma(J)})>0.
	\end{equation}
	 Moreover, by part \eqref{it4} of Theorem \ref{bigtheo}, we have $\tilde{\nu}=\mu_{\sigma(J)}$. 
	
	Now, consider the constant function $$h(z):=\gamma(z)-g_{\sigma(J)}(z)= -A(\Omega_{c_{J}})+\log{\mathrm{Cap}(\sigma(J))}$$ on $\mathbb{C}^+$. Since $g_{\sigma(J^{\prime})}$ can be as close to $0$ as we wish and $\gamma>0$, $h$ should be non-negative on $\mathbb{C}^+$. This gives 
	\begin{equation}\label{yard2}
		\exp^{A(\Omega_{c_{J}})}\leq\mathrm{Cap}(\sigma(J)).
	\end{equation}
 By \eqref{yard} and \eqref{yard2}, $h(z)=0$ and $\gamma=g_{\sigma(J)}$ on $\mathbb{C}^+$. Since $\lim_{y\rightarrow 0^+} g_{\sigma(J)}(x+iy)=0$ quasi-everywhere on $\sigma(J)$ (see e.g. p. 53-54 in \cite{saff}) we have $\gamma=0$ quasi-everywhere on $\sigma(J)$. 
 
 By \eqref{yard} and \eqref{yard2}, $A(\Omega_{c_{J}})=\log{\mathrm{Cap} (\sigma(J))}.$ Thus  $\lim_{n\rightarrow\infty} (a_1\cdots a_n)^{1/n}= \mathrm{Cap} (\sigma(J))$ for $m_{\Omega_{c_J}}$ almost every $c_{J^\prime}$. This proves the last statement.
\end{proof}

Due to Kotani's theory, $\gamma=0$ on the spectrum has some consequences. For this, we need to discuss another concept. If we are given $J$, then $J$ is called reflectionless on $A\subset \mathbb{R}$ if 
$$\lim_{\epsilon\rightarrow 0^+} \langle \delta_n, (J-z-i\epsilon)^{-1}\delta_n\rangle=0$$
for almost everywhere $z\in A$ and for all $n\in\mathbb{Z}$. 

If $J(a_n,b_n)_{n=-\infty}^\infty$ is almost periodic, $A(\Omega_{c_J})>-\infty$ and $\gamma=0$ on $\sigma(J)$ quasi-everywhere then for $\mu_{\Omega_{c_J}}$ almost every $c_{J^\prime}$, $J^\prime$ is reflectionless on $\sigma (J)$, see \cite{kotani}, \cite{anand}. Moreover if $|\sigma(J)|>0$ then being reflectionless implies the existence of a non-trivial absolutely continuous part. In this case, we have $inf_{n\in\mathbb{Z}} a_n>0$ by the following argument: Let us assume that $inf_{n\in\mathbb{Z}} a_n=0$. Then by part \eqref{it2} of Theorem \ref{bigtheo}, $\liminf_{n\rightarrow \infty} a^\prime_n = \liminf_{n\rightarrow -\infty} a^\prime_n=0$ for almost every $c_{J^\prime(a_n^\prime,b_n^\prime)_{n=-\infty}^\infty}$ in $\Omega_{c_J}$. This implies by \cite{dombrowski} that, $(J^\prime)^+$ and $(J^\prime)^-$ does not have an absolutely continuous spectrum for almost every $c_{J^\prime(a_n^\prime,b_n^\prime)_{n=-\infty}^\infty}$ in $\Omega_{c_J}$. Since $\sigma_{\mathrm{ac}}(J^\prime)= \sigma_{\mathrm{ac}}((J^\prime)^+)\cup \sigma_{\mathrm{ac}}((J^\prime)^-)$ holds true (see e.g. Lemma 3.11 in \cite{teschl}) for almost every $c_{J^\prime(a_n^\prime,b_n^\prime)_{n=-\infty}^\infty}$, $J^\prime$ has no absolutely continuous part which contradicts with being reflectionless. Now let us sketch the proof of a result which is a generalization of Proposition 4.1 (we skip one of the equivalent statements) in \cite{kruger}. The proof is almost identical. 
\begin{theorem}\label{dct}
Let $K$ be a Parreau-Widom set such that DCT holds on $\overline{\mathbb{C}}\setminus K$ and $J$ be a Jacobi operator with $\sigma(J)=K$. Then the following conditions are equivalent.
\begin{enumerate}[(i)]
\item J is reflectionless on $K$.
\item $J$ is almost periodic and $J^+$ and $J^-$ are regular.
\end{enumerate} 
Two sided Jacobi operators satifying these conditions are called the isospectral torus of $K$ which we denote by $\mathrm{IT}(K)$.
\end{theorem}
\begin{proof}
(i)$\implies$ (ii) Almost periodicity of reflectionless operators in this case is a result of Sodin and Yuditskii \cite{sodin}, see also Section 1 in \cite{christi}. Regularity of $J^+$ and $J^-$ can be found in Corollary 2.1 in \cite{christi}.

(ii)$\implies$ (i) By Proposition \ref{prop}, we have $\gamma=0$ almost everywhere on $K$. Then the proof is the same with the proof of (2)$\implies$ (1) in Proposition 4.1 of \cite{kruger}. We note that Remling's Theorem (Theorem 1.1 in \cite{christi} and see also Theorem 1.4 in \cite{remling}) is applicable on Parreau-Widom sets since \eqref{ess} holds.
\end{proof}

There are Parreau-Widom sets such that the equivalence in Theorem \ref{dct} does not hold. The following remarkable result was proven by Volberg and Yudistkii, see Theorem 1.7 and Theorem 1.8 in \cite{volberg}.

\begin{theorem}\label{volyud}
	Let $K=[b_0,a_0]\setminus \cup_{j\geq 1} (e_j,f_j)$ be a Parreau-Widom set such that DCT does not hold on $\overline{\mathbb{C}}\setminus K$. Suppose that the following conditions hold:
	\begin{enumerate}[(i)]
		\item Every reflectionless operator $J$ on $K$ with $\sigma(J)=K$ has purely absolutely continuous spectrum.
		\item The frequencies $\{\mu_K([b_0, b_j])\}_{j\geq 1}$ are rationally independent.
	\end{enumerate}
	Then none of the reflectionless Jacobi operators on $K$ whose spectrum is $K$ are almost periodic.
\end{theorem}

Examples of Parreau-Widom sets satisfying the hypothesis of Theorem \ref{volyud} can be found in Section 1.3 of \cite{volberg}. The following result is an immediate corollary of Theorem \ref{volyud}.

\begin{theorem}\label{IT}
	Let $K$ be a Parreau-Widom set satisfying all assumptions of Theorem \ref{volyud}. Then there is no $J$ with $\sigma(J)=K$ such that it is almost periodic and $J^+$ is regular.
\end{theorem}
\begin{proof}
	Assume to the contrary that such a $J$ exists. Then $\gamma(z)=0$ almost everywhere (quasi-everywhere implies almost everywhere) on $K$ by Proposition \ref{prop}. Then by Kotani's theory, this implies that for $\mu_{\Omega_{c_J}}$ almost every $c_{J^\prime}$, $J^\prime$ is reflectionless on $\sigma (J)$. By part \eqref{it4} of Theorem \ref{bigtheo}, $\sigma(J)= \sigma(J^\prime)$ holds for each $J^\prime$ associated with an element in $\Omega_{c_J}$. This is impossible by Theorem \ref{volyud} since all Jacobi operators associated with an element in $\Omega_{c_J}$ are almost periodic.  
\end{proof}
Comparing how different Jacobi operators may behave depending on DCT, the sets that satisfy the assumptions in Theorem \ref{volyud} can be considered as the representative of the bad sets in terms of spectral theory of orthogonal polynomials.
\section{Szeg\H{o} Class}

A measure $\mu$ can be written as $$d\mu(x)=f(x)dx+d\mu_{\mathrm{s}}(x)$$ by Lebesgue's decomposition theorem where f is the absolutely continuous part and $d\mu_{\mathrm{s}}$ denotes the singular part with respect to the Lebesgue measure. If $\mathrm{Cap(supp(\mu))}>0$, then $\|p_{n}(\cdot;\mu)\|_{L^2(\mu)}/\mathrm{Cap(supp(\mu))}^n$ is well defined and we denote this ratio by $W_n(\mu)$. 

Following \cite{christi}, let us define the Szeg\H{o} class of measures on a given Parreau-Widom set $K$. By $\mathrm{ess\,supp}(\cdot)$ we denote the set of accumulation points of the support. A measure $\mu$ is in the Szeg\H{o} class of $K$ if 

\begin{enumerate}[(i)]
	\item $\mathrm{ess\,supp}(\mu)=K.$
	\item $\int_K \log{f(x)}\,d\mu_K(x)>-\infty.$
	\item the isolated points $\{x_n\}$ of $\mathrm{supp}(\mu)$
	satisfy $\sum_k g_K(x_n)<\infty.$
\end{enumerate}

By Theorem 2 in \cite{christiansen} and its proof, (ii) can be replaced by one of the following conditions on the recurrence coefficients associated with $\mu$ so that the definition includes the same family of measures:
\begin{enumerate}[(ii$^\prime$)]
	\item $\limsup_{n\rightarrow\infty}W_n(\mu)>0.$
	\end{enumerate}
\begin{enumerate}[(ii$^{\prime\prime}$)]
	\item $\liminf_{n\rightarrow\infty}W_n(\mu)>0.$
	\end{enumerate}
We denote the Szeg\H{o} class of $K$ by $\mathrm{Sz}(K)$. By the above definition, in particular by (ii$^\prime$), $\mu\in\mathrm{Sz}(K)$ implies that $\mu$ is regular in the sense of Stahl-Totik. 

A Jacobi operator $J^+(a_n,b_n)_{n=1}^\infty$ is called asymptotically almost periodic if there is an almost periodic Jacobi operator $J (a_n^\prime,b_n^\prime)_{n=-\infty}^\infty$ such that $\limsup_{n\rightarrow\infty}(|a_n-a_n^\prime|+ |b_n-b_n^\prime|)=0.$ In this case, we call  $J (a_n^\prime,b_n^\prime)_{n=-\infty}^\infty$ the almost periodic limit.

Let $K$ be a Parreau-Widom set such that DCT holds on $\overline{\mathbb{C}}\setminus K$. If $\mu\in\mathrm{Sz}(K)$ then by Theorem 1.2 in \cite{christi} $J^+(\mu)$ is asymptotically almost periodic. By Theorem 2.1 in \cite{christi}, if $J\in\mathrm{IT}(K)$ then the spectral measure of $J^+$ belongs to the Szeg\H{o} class of $K$. 

For the Szeg\H{o} class of Parreau-Widom set, we have a result similar to Theorem \ref{IT}.

\begin{theorem}\label{szeg}
	Let $K$ be a Parreau-Widom set satisfying all assumptions of Theorem \ref{volyud}. If $\mu\in\mathrm{Sz}(K)$ then $J^+(\mu)$ cannot be asymptotically almost periodic. 
\end{theorem}
\begin{proof}
	Assume that there is a measure $\mu\in \mathrm{Sz}(K)$ such that $J^+(\mu)=J^+(a_n,b_n)_{n=1}^\infty$ is asymptotically almost periodic   with the almost periodic limit $J^\prime(a_n^\prime,b_n^\prime)_{n=-\infty}^\infty$.
	
	 Since $\mu$ has a non-trivial absolutely continuous part, by \cite{dombrowski} there is a $d>0$ such that $\inf_{n\in\mathbb{N}}{a_n} =d$. This implies that $\liminf_{n\rightarrow\infty}a_n^\prime \geq d$. By part \eqref{it2} of Theorem \ref{bigtheo}, we also have $inf_{n\in\mathbb{Z}}{a_n^\prime}\geq d>0$. Thus $(J^\prime)^+$ is well defined. Moreover, by Theorem 7.1 (f), there is a positive number $A$ such that 
	 
	 \begin{equation}A= \label{eq1}
	 \lim_{n\rightarrow\infty}(a_1^\prime\cdots a_n^\prime)^{1/n}.
	 \end{equation}
	 
	 By Lemma 3.9 in \cite{teschl}, $J^+(\mu)$ and $(J^\prime)^+$ have the same essential spectra which implies that $\sigma(J^\prime)=K$ by part \eqref{it4} of Theorem \ref{bigtheo}. Now, let us show that $(J^\prime)^+$ is regular which contradicts with Theorem \ref{IT}. 
	 
	 By regularity of $\mu$ and \eqref{eq1}, there is a $C>0$ such that 
	 \begin{equation}\label{eq2}
	  C=\lim_{n\rightarrow\infty}\left( \frac{a_1\cdots a_n}{a_1^\prime\cdots a_n^\prime}\right)^{1/n}.
	 \end{equation}
	 If we let $r_n:=a_n/a_n^\prime$ for $n\in\mathbb{N}$ in \eqref{eq2} and taking logarithm of both sides then we obtain 
	 \begin{equation*}
	   \log{C}=\lim_{k\rightarrow\infty} \frac{1}{k}\sum_{n=1}^k \log{r_n}.
	 \end{equation*}
	 But since $\log{r_n}\rightarrow 0$ as $n\rightarrow\infty$ and the partial sums $\frac{1}{k}\sum_{n=1}^k \log{r_n}$, for k=1,\dots, are Cesaro means of $(\log{r_n})$ we have $\log{C}=0$ and thus $C=1$. Since $J^+(\mu)$ and $(J^\prime)^+$ have the same essential spectra, $C=1$ implies the regularity of the latter.
		
\end{proof}
We should mention that for all Parreau-Widom sets $K$, $\mu_K\in\mathrm{Sz}(K)$. The part (i) of the definition of the Szeg\H{o} class above holds for $\mu_K$ as mentioned in Section 2. The proof of part (iii) is straightforward for $\mu_K$ can not contain point mass since this would imply that $\mathrm{Cap}(K)=0$. For the proof of (ii) for $\mu_K$, see e.g. Section 4 in \cite{christiansen}. Therefore, by Theorem \ref{szeg}, there are some Parreau-Widom sets $K$ for which $J^+(\mu_K)$ is not asymptotically almost periodic.

\section{The Szeg\H{o} class and the isospectral torus of a generic set}

Before surveying the known examples for zero measure case, we first give definitions of Szeg\H{o} class and isospectral torus on generic sets (we discuss below what we mean by a generic set). For previous suggestions for the definition of isospectral torus, we refer the reader to \cite{kruger}, \cite{mant3}. Our approach will be very similar to that of \cite{kruger}. 

Our definitions should be compatible with the definitions in Parreau-Widom case. For this, we impose the conditions of non-polarity and regularity with respect to the Dirichlet problem to the sets we consider. The sets satisfying these mild assumptions are our generic sets. For the Szeg\H{o} class we suggest the following definition.

\begin{definition}
	Let $K$ be a non-polar compact subset of $\mathbb{R}$ which is regular with respect to the Dirichlet problem. Then $\mu\in\mathrm{Sz}(K)$ if
	
	\begin{enumerate}[(i)]
		\item $\mathrm{ess\,supp}(\mu)=K.$
		\item $\inf_{n\in\mathbb{N}}W_n(\mu)>0.$
		\item the isolated points $\{x_n\}$ of $\mathrm{supp}(\mu)$
		satisfy $\sum_k g_K(x_n)<\infty.$
	\end{enumerate}	
\end{definition}

We remark that, if $K$ is taken to be equal to a Parreau-Widom set then this definition coincides with the definition given in Section 3. We choose the property $\inf_{n\in\mathbb{N}} W_n(\mu)>0$ instead of $\sup_{n\in\mathbb{N}} W_n(\mu)>0$ in order to have a smaller class of measures. 

In \cite{alpeq}, the following result was proven:

\begin{theorem}\label{eqin}
	Let $K$ be a non-polar subset of $\mathbb{R}$. Then $\inf_{n\in\mathbb{N}}W_n(\mu_K)\geq 1$.
\end{theorem}
As a corollary we have,

\begin{corollary}\label{eqcor}
	Let $K$ be a non-polar subset of $\mathbb{R}$ which is regular with respect to the Dirichlet problem. Then $\mu_K\in\mathrm{Sz}(K)$.
\end{corollary}
 \begin{proof}
 	Since $K$ is regular, combining Theorem 4.2.3 in \cite{Ransford} and Theorem 5.5.13 in \cite{Sim3} we obtain $\mathrm{supp}(\mu_K)= K$. The condition $\inf_{n\in\mathbb{N}}W_n(\mu)>0$ holds by Theorem \ref{eqin}. The condition concerning isolated points automatically holds since $\mu_K$ does not have isolated points.
 \end{proof}
This result implies that the Szeg\H{o} class of a generic set is always non-empty. 

Next, let us discuss the isospectral torus.

\begin{definition}
	Let $K$ be a non-polar compact subset of $\mathbb{R}$ which is regular with respect to the Dirichlet problem. Then a two-sided Jacobi operator  $J$ is in $\mathrm{IT}(K)$ if
	\begin{enumerate}[(i)]
		\item $\sigma_{\mathrm{ess}}(J)=K.$
		\item $J$ is almost periodic.
		\item There is a $c_{J^\prime}\in \Omega_{c_J}$ such that $(J^\prime)^+$ is regular in the sense of Stahl-Totik.
	\end{enumerate}	
\end{definition}
 This definition is almost identical with the definition suggested in p. 83 of \cite{kruger}. Our definition allows that $a_n=0$ for some $n$. Thus, there may be some $c_{\tilde{J}}\in\Omega_{c_J}$ such that $\left(\tilde{J}\right)^+$ is not regular. Nevertheless by Proposition \ref{prop} for $m_{\Omega_{c_J}}$ almost every $c_{J^\prime}$, we have that $J^\prime$ is two-sided regular with $a_n \neq 0$ for all $n\in\mathbb{N}$. Here, the emphasis is on the hull rather than each element of the hull. This will allow us to cover some interesting examples (see Section 5) of operators in the definition of the isospectral torus which
 are not covered by the definition in \cite{kruger}. 
 
 We remark that this definition coincides with the definition discussed in Section 2, as easy to check, if $K$ is taken to be a finite gap set. Moreover, $\mathrm{IT}(K)=\emptyset$ by Theorem \ref{IT} if $K$ satisfies the assumptions of Theorem \ref{volyud}.
 
 There are some interesting examples of almost periodic Jacobi operators discussed in the literature many times, see e.g. \cite{avil}. In our discussion of the concepts of Szeg\H{o} class and isospectral torus, we fix the set first instead of choosing an operator. Hence, we disregard these kind of examples. We focus on three cases in the next section: Cantor ternary set, polynomial Julia sets and generalized polynomial Julia sets. 
 
 \section{Three examples}
 \subsection{Cantor ternary set}
 For a general treatment of iterated function systems, we refer the reader to \cite{barnben}, \cite{hut}. 
 
 Let $w_1(x)= x/3$ and $w_2(x)= (x+2)/3.$ Then the Cantor ternary set $K_0$ is the unique solution of $K=\cup_{j\in{1,2}} w_j(K)$ among the subsets of $[-1,1]$. The Cantor-Lebesgue measure $\eta_{K_0}$ is the unique unit Borel measure satisfying 
 \begin{equation*}
 \int_{K_0} f\, d\mu= \frac{1}{2}\sum_{j=1}^{2} \int_{K_0} (f\circ w_j)\, d\mu
 \end{equation*}
 for all $f\in C(K_0)$. We have $\mathrm{supp}(\eta_{K_0})= K_0$ and $|K_0|=0$. It is easy to verify that $\eta_{K_0}$ is purely singular continuous.
 
 In \cite{kruger} scaled and translated versions of $K_0$ and $\eta_{K_0}$ were under investigation. We do not make any distinctions below for the usual Cantor set and scaled and translated versions since these results are invariant under such operations. 
 
 Some results concerning the moments $\int t^n\, d{\eta_{K_0}}(t)$ can be found in \cite{barnben}. A numerically stable algorithm for calculating the recurrence coefficients for $\eta_{K_0}$ was given in \cite{mant1} by Mantica. In \cite{Heilman} some properties  regarding orthogonal polynomials associated with $\eta_{K_0}$ were tested numerically. It was conjectured in \cite{mant2} that these coefficients are asymptotically almost periodic. This conjecture was repeated in \cite{kruger} (Conjecture 3.1) by checking the behavior of the first $100000$ coefficients.
 
 It is due to Bia\l as-Cie\.{z} and Volberg that the Cantor ternary set is regular with respect to the Dirichlet problem, see \cite{bial}. Regularity of $\eta_{K_0}$ in the sense of Stahl-Totik was proven in \cite{kruger}. If the conjecture on the recurrence coefficients regarding almost periodicity is correct then there is an almost periodic Jacobi operator $J$ (which is the almost periodic limit of $J^+(\eta_{K_0})$) such that $J^+$ is regular and $\sigma_{\mathrm{ess}}(J)= K_0$. Thus, we have $\mathrm{IT}(K_0)\neq \emptyset$.
 
 In \cite{kruger}, the behavior of $(W_n(\mu))_n$ was numerically examined. It was conjectured that (Conjecture 3.2) 
 \begin{equation}\label{widcon}
 0<\inf_{n\in\mathbb{N}} W_n(\eta_{K_0})\leq \sup_{n\in\mathbb{N}} W_n(\eta_{K_0})<\infty.
 \end{equation}
 Hence $\eta_{K_0}\in \mathrm{Sz}(K_0)$ provided that \eqref{widcon} holds.  
 
 By \cite{kruger}, for $\beta>0$, the spectral measure for $J^+(\eta_{K_0})+\beta \langle \delta_1,\cdot\rangle \delta_1$ is purely discrete. This shows that, unlike  the absolutely continuous part (see e.g. Section 7 in \cite{Sim3}), the singular part of the spectral measure of a Jacobi operator is not preserved under finite rank perturbations.

\subsection{Polynomial Julia sets} For some of the results on the orthogonal polynomials associated with equilibrium measures, see \cite{Barnsley3}, \cite{Barnsley4}, \cite{bes2}. There are some other results when the measure is not the equilibrium measure, see e.g. \cite{besger}, \cite{knill}.

For a given non-linear complex polynomial $f$, the Julia set $H(f)$ can be defined in $\overbar{\mathbb{C}}$ as $\partial \{z\in\overbar{\mathbb{C}}: f^{(n)}(z)\rightarrow\infty \mbox{ locally uniformly in } \overbar{\mathbb{C}}\} $ where $f^{(n)}$ is the $n$-th iteration of $f$. Note that $H(f)$ is a non-polar compact subset of $\mathbb{C}$. Regularity of $H(f)$ with respect to the Dirichlet problem was proved in \cite{Mane}.

 We want to consider the simplest case that is $f(z)=z^2-c$ with $c>2$. For such an $f$, $H(f)$ is a Cantor set on $\mathbb{R}$ and $|H(f)|=0$, see e.g. \cite{brolin}. The recurrence coefficients for $\mu_{H(f)}$ can be calculated recursively by the following formulas, see p. 89 of \cite{bes2}:
\begin{align*}
 a_1 &=\sqrt{c},  n\geq 1 \\
 a_{2n}^2 a_{2n-1}^2 &= a_n^2, \,\,\\
 a_{2n}^2+ a_{2n+1}^2&= c \\
 b_n&=0.
\end{align*}
Let $c\geq 3$. Then $J^+\left(\mu_{H(f)}\right)$ is almost periodic, see p. 92 of \cite{bes2}. Note that $\tilde{\nu}=\mu_{H(f)}$ by regularity of $\mu_{H(f)}$, see e.g. Theorem \ref{eqin}. Since $H(f)$ is regular with respect to the Dirichlet problem, $\mathrm{supp}\left(\mu_{H(f)}\right)=H(f)$. Let us denote the almost periodic extension of this operator to $l^2(\mathbb{Z})$ by $J(\mu_{H(f)})$. In this extension we have $a_0=0$. Some properties of $J^-(\mu_{H(f)})$ were studied in \cite{sodin2}. We have $\sigma(J(\mu_{H(f)}))= H(f)$ by part \eqref{it4} of Theorem \ref{bigtheo}. Since $J^+\left(\mu_{H(f)}\right)$ is regular in the sense of Stahl-Totik we have $J(\mu_{H(f)})\in \mathrm{IT}(H(f))$. By Corollary \ref{eqcor}, we also have $\mu_{H(f)}\in \mathrm{Sz}(H(f))$. 

Another interesting property of this family of Jacobi operators is that the spectral measure for $J^+_{\restriction_1}\left(\mu_{H(f)}\right)$ is purely discrete, see Section 2 in \cite{Barnsley4}. Thus, the singular continuous part of the spectrum is not preserved under such an operation. We remark that if $\mu$ is a measure with a non-trivial absolutely continuous part, then absolutely continuous spectra of $J^+(\mu)$ and $J^+_{\restriction_1}(\mu)$ coincide, see Chapter 7 in \cite{Sim3}.

\subsection{Generalized polynomial Julia sets}
For background information about generalized Julia sets, we refer the reader to \cite{Bruck} and \cite{Fornaess}. In \cite{alpgon2}, the authors restrict their attention to Julia sets from Section 4 of \cite{Bruck}. 

\begin{definition}Let $f_n(z)=\sum_{j=0}^{d_n}a_{n,j}\cdot z^j$ where $d_n\geq 2$ and $a_{n,d_n}\neq 0$ for all $n\in\mathbb{N}$. We say that $(f_n)$ is a \emph{regular polynomial sequence} if the following properties are satisfied:
	\begin{itemize}
		\item There exists a real number $A_1> 0$ such that $|a_{n,d_n}|\geq A_1$, for all $n\in\mathbb{N}$.
		\item There exists a real number $A_2\geq 0$ such that $|a_{n,j}|\leq A_2 |a_{n,d_n}|$ for $j=0,1,\ldots, d_n-1$ and $n\in\mathbb{N}$.
		\item There exists a real number $A_3$ such that $$\log{|a_{n,d_n}|}\leq A_3\cdot d_n,$$
		for all $n\in\mathbb{N}$. 
	\end{itemize}
\end{definition}

Let $F_l(z):=(f_l\circ\ldots\circ f_1)(z)$. If $(f_n)$ is a regular polynomial sequence then the Julia set $H_{(f_n)}$ associated with $(f_n)$ is defined as $\partial\{z\in\overline{\mathbb{C}}: F_n(z) \mbox{ goes locally uniformly to } \infty \}$. These Julia sets are regular with respect to the Dirichlet problem, see \cite{Bruck}.

Orthogonal polynomials for a special family of generalized Julia sets were studied in detail in \cite{g1}, \cite{alpgon}, \cite{alpgon2}, \cite{alp3}. The construction is from \cite{gonc}. Let $\gamma=(\gamma_s)_{s=1}^\infty$ be a sequence such that $0<c<\gamma_s<1/4$ for some $c$. Define $(f_n)_{n=1}^\infty$ by $f_1(z):=2z(z-1)/\gamma_1+1$ and $f_n(z):=z^2/(2\gamma_n)+1-1/(2\gamma_n)$ for $n>1$. Then, $K(\gamma):= H_{(f_n)}$. For each $\gamma$ we obtain a different set. We remark that $K(\gamma)$ is a non-polar Cantor set lying on $\mathbb{R}$. 

Depending on numerical evidence, it was conjectured (Conjecture 3.3) in \cite{alp3} that $J^+({\mu_{K(\gamma)}})$ is asymptotically almost periodic for all $\gamma$. If this conjecture is correct, then $\mathrm{IT}(K(\gamma))\neq \emptyset$ by a similar argument used in Section 5.1 in order to show $\mathrm{IT}(K_0)\neq \emptyset$. 

\section{Further discussion and some open problems}
    \begin{itemize} 
	\item Theorem \ref{IT} implies the existence of positive measure sets with empty isospectral torus. In Section 5.2, a zero Lebesgue measure set with non-empty isospectral torus is included. Is there a zero measure non-polar compact subset $K$ of $\mathbb{R}$ which is regular respect to the Dirichlet problem satisfying $\mathrm{IT}(K)=\emptyset$? If this is the case, is there a general condition on zero measure sets which is similar to DCT condition?
	
	\item What measures other than the equilibrium measure belong to the Szeg\H{o} class of a generic set? Is there a generic set $K$ with $|K|=0$ such that $J^+(\mu_K)$ is not asymptotically almost periodic? If there is, does it imply that $\mathrm{IT}(K)=\emptyset$?
	
    \item This problem is also mentioned in \cite{alpeq}. What is the value of $\liminf {a_n}$ where $(a_n)_{n=1}^\infty$ is the sequence of recurrence coefficients for $\mu_{K_0}$. If this value is $0$ then by Corollary 1 in \cite{alpeq}, $(W_n(\mu_{K_0}))_{n=1}^\infty$ is unbounded. It may be also true that  $\liminf_{n\rightarrow\infty} {a_n}\neq 0$ but $(W_n(\mu_{K_0}))_{n=1}^\infty$ is unbounded
    
    A couple of algorithms for computing the recurrence coefficients associated with $\mu_{K_0}$ were already discussed in \cite{mant4}. Studying these coefficients, at least numerically, can give some ideas about general behavior of $(W_n(\mu_K))_{n=1}^\infty$ and $J^+({\mu_K})$ for generic zero Lebesgue measure sets. Moreover, a comparison of these results with the results obtained for the Cantor measure in \cite{Heilman}, \cite{kruger}, \cite{mant1} is of particular interest since $\mu_{K_0}$ and $\eta_{K_0}$ are mutually singular by \cite{maka}.

	\end{itemize}
%\begin{acknowledgements}
%\end{acknowledgements}



\end{document}